\documentclass[11pt]{amsart}

\usepackage{amsmath, amssymb, amsfonts, amsthm, mathrsfs, mathtools}
\usepackage{thmtools}

\usepackage[hcentering, vcentering, total={5.9in, 8.2in}]{geometry} 

\usepackage{graphicx}   
\usepackage{tikz-cd}    
\usetikzlibrary{cd}

\tikzset{
    symbol/.style={
    draw=none,
    every to/.append style={
    edge node={node [sloped, allow upside down, auto=false]{$#1$}}
    },
},
}

\usepackage{comment}

\usepackage{enumerate}
\usepackage{enumitem}

\usepackage{fancyhdr}

\usepackage{xcolor}

\usepackage[colorlinks=true, linkcolor=blue, citecolor=red]{hyperref}

\theoremstyle{plain}
\newtheorem{theorem}{Theorem}[section]
\newtheorem{lemma}[theorem]{Lemma}
\newtheorem{proposition}[theorem]{Proposition}
\newtheorem{corollary}[theorem]{Corollary}

\newtheorem{claim}[theorem]{Claim}

\theoremstyle{definition}
\declaretheorem[style=definition,qed=$\triangle$,sibling=theorem]{definition}

\declaretheorem[style=definition,qed=$\triangle$,sibling=theorem]{remark}
\declaretheorem[style=definition,qed=$\triangle$,sibling=theorem]{question}

\theoremstyle{plain}

\usepackage[nameinlink,capitalize]{cleveref}  
\Crefname{question}{Question}{Questions}
\Crefname{claim}{Claim}{Claims}

\def\R{\mathbb{R}}
\def\C{\mathbb{C}}
\def\Z{\mathbb{Z}}

\def\S{\mathbb{S}}

\def\D{\mathcal{D}}

\def\d{\,\mathrm{d}}

\DeclareMathOperator{\Id}{Id}

\DeclareMathOperator{\img}{Img}
\DeclareMathOperator{\Diff}{Diff}

\def\Imm{\mathfrak{I}\mathrm{mm}}
\def\Conv{\mathfrak{C}\mathrm{onv}}

\newcommand{\defi}{\coloneqq}

\newcommand{\st}{\, | \,}

\def\Op{\mathcal{O}p\,}
\def\std{\mathrm{std}}

\def\prol{\mathrm{prol}}
\def\derprol{\mathrm{der-prol}}

\title{Engel CR submanifolds of $\mathbb{C}^3$}

\author{Eduardo Fern\'andez}
\address{Department of Mathematics\\ University of Georgia\\ Athens\\ GA, USA}
\email{eduardofernandez@uga.edu}

\author{\'Alvaro del Pino}
\address{Utrecht University, Department of Mathematics, Budapestlaan 6, 3584 CD Utrecht, The Netherlands}
\email{a.delpinogomez@uu.nl}

\author{Wei Zhou}
\address{Instituto de Ciencias Matemáticas, Consejo Superior de Investigaciones Científicas, Madrid, Spain}

\address{Departmento de Álgebra, Geometría y Topología, Facultad de Ciencias Matemáticas, Universidad Complutense de Madrid, Madrid, Spain}

\email{wzhou02@ucm.es}
\date{\today}

\begin{document}
\begin{abstract}
We give a sufficient condition for an $\mathbb{S}^1$-bundle over a $3$-manifold to admit an immersion (or embedding) into $\mathbb{C}^3$ so that its complex tangencies define an Engel structure. In particular, every oriented $\mathbb{S}^1$-bundle over a closed, oriented $3$-manifold admits such an immersion. If the bundle is trivial, this immersion can be chosen to be an embedding and, moreover, infinitely many pairwise smoothly non-isotopic embeddings of this type can be constructed. 

These are the first examples of compact submanifolds of $\mathbb{C}^3$ whose complex tangencies are Engel, answering a question of Y. Eliashberg.
\end{abstract}
\maketitle

\section{Introduction}

An Engel structure is a maximally non-integrable rank-2 distribution $\D \subseteq TN$ on a $4$\nobreakdash-manifold $N$. Just like contact structures, Engel structures are \emph{stable} \cite{VershikGershkovich1987,Montgomery1993}: they are given by an open partial differential relation and they admit a unique local normal form.

This lack of local invariants led to the conjecture that Engel structures should admit a topological theory, possibly abiding by the $h$-principle \cite[Intrigue]{cieliebak2024introduction}. A first positive result in this direction was due to Vogel \cite{VogelEngel}, who constructed an Engel structure in each parallelisable $4$-manifold. More recently, a series of results established the $h$-principle for various classes of Engel structures \cite{CPPP,CPP2020loose,PinoVogel_EngelLutz} and for various kinds of submanifolds within them \cite{AdachiLoops,GeigesLoops,CasalsPino_Engelknots,delPinoPresas2019,FMP_FundamentalHorizontal,gompf2022transverse,Kegel_NonIsotopicTransverseTori}. Engel structures have also been explored from a wide range of geometric perspectives \cite{Montgomery1999,PresasSolaConde2014,BeschastnyiMedvedev2016,Mitsumatsu2018,KotschickVogel2018,Yamazaki2019,Pia2022RiemannianEngel}.

In 2017, a workshop on Engel Topology was organised at AIM \cite{AIMEngel2017} and a list of open problems was assembled \cite{AIMEngel2017B}. Many of these went along the following line of research: \emph{What is the relationship between Engel and complex geometry?} Even though there has been progress in this direction \cite{Zhao2018Thesis,RuiNicola_ExoticHolEngel,NicolaGiovanni_EngelOnComplexSurface}, the following question, due to Y. Eliashberg, has remained open:
\begin{question} \label{ques:main}
Does there exist a compact $4$-dimensional submanifold $N \subseteq \C^3$ whose complex tangencies $\D_J := TN\cap J(TN)$ form an Engel distribution?    
\end{question}

From the viewpoint of Contact Topology, this line of reasoning is very natural: If $Y^{2n-1}\subseteq \C^n$ is a smooth real hypersurface, its complex tangencies $\xi_{J_\std} \defi TY\cap J_\std (TY)$ define a corank-one distribution that is contact precisely when $Y$ is strictly pseudoconvex; moreover, in this case, $(Y,\xi_{J_\std})$ is tight \cite{eliashbergFillingHolomorphic1990,Niederkrueger2006_plastikstufe, Presas2007_nonfillable,NiederkruegerVanKoert2007, cieliebakSteinWeinsteinBack2012,bormanExistenceClassificationOvertwisted2015}.

In this paper we show that the answer to  \cref{ques:main} is positive. This leads to the following, still very much open, question: can some form of \emph{rigidity} be established in Engel Topology using this relation to Complex Geometry?

\subsection{Statement of results}

Recall that $N^4 \subseteq \C^3$ is \emph{co-real} (also called \emph{CR-regular}) if $\D_J$ is a plane distribution. Remarkably, Gromov proved an $h$-principle for co-real embeddings \cite{gromovPartialDifferentialRelations1986, cieliebak2024introduction}, meaning that this condition is purely formal: $N$ admits a co-real embedding into $\C^3$ if and only if $N$ is parallelisable \cite{slapar2015cr}. This perfectly matches the fact that a closed orientable 4-manifold carries an orientable Engel structure precisely when it is parallelisable \cite{VogelEngel}. As such, from an algebraic topology viewpoint, these two geometric situations are equivalent.

More generally, one can consider co-real immersions $g: N^4 \looparrowright (X^6,J)$ of a $4$-manifold into an almost complex manifold of real dimension $6$. Co-reality means that the complex tangencies
\begin{equation*}
    \D^g_J\defi g_*^{-1}\big( g_*TN\cap J(g_*TN) \big)
\end{equation*}
have rank $2$.

In this article we focus on $4$-manifolds $N$ given as oriented $\S^1$-bundles over oriented $3$\nobreakdash-manifolds $M^3$. For these, we provide a sufficient condition to admit a co-real embedding/immersion into $(X,J)$ so that the complex tangencies are Engel. This sufficient condition is the existence of a \emph{convex bundle embedding/immersion} $g: N \looparrowright TM$; i.e., a bundle map that is fibrewise embedded/immersed and whose fibrewise derivative, when normalised, traces a convex curve on the sphere (see \cref{def:convex_der_curves} for further details). Our main result reads as follows:

\begin{theorem}\label{thm:main_thm}
    Let $N$ be an oriented $\S^1$-bundle over a closed, oriented $3$-manifold $M$ and let $g:N \looparrowright TM$ be a convex bundle immersion. Then, for any totally real immersion $\iota:M\looparrowright(X,J)$, into an almost complex manifold of real dimension $6$, there exists a co\nobreakdash-real immersion $\tilde{g}:N\looparrowright (X,J)$ whose complex tangencies $\D_{J}^{\tilde{g}}$ form an Engel distribution. Moreover, if $g$ and $\iota$ are embeddings, then $\tilde{g}$ is also an embedding.
\end{theorem}

The proof of \cref{thm:main_thm} is based on a zooming-in argument, which is somewhat reminiscent of Donaldson's zooming argument used in the construction of symplectic divisors \cite{donaldsonSymplecticSubmanifolds1996} and of the corrugation phenomenon in $h$-principle \cite{Smale1959,Theilliere2019}.

\begin{remark}
Recall that a submanifold of an almost complex manifold is \emph{totally real} when its tangent space contains no complex lines. Again, Gromov proved an $h$-principle for these, so the obstruction to its existence is purely formal \cite{gromovPartialDifferentialRelations1986,cieliebak2024introduction}. In particular, Forstnerič proved that a totally real embedding $\iota:M^3\hookrightarrow(X^6,J)$, as in \cref{thm:main_thm}, always exists \cite{forstnericTotallyRealEmbeddings1986}. 
\end{remark}

\cref{thm:main_thm} yields a broad class of $4$-manifolds that can be immersed, and infinitely of them embedded, into $\C^3$, so that the complex tangencies define an Engel structure:
\begin{theorem}\label{thm:main_result}
    Let $N$ be an oriented $\S^1$-bundle over a closed, oriented $3$-manifold $M$. Then, there exists a co-real immersion $g:N\looparrowright\C^3$ whose complex tangencies $\D^g_{J_\std}$ form an Engel distribution. If $N$ is the trivial bundle, then $g$ can be chosen to be an embedding.
\end{theorem}

We conjecture that $g$ can also be chosen to be an embedding in the case that $N$ has even Euler class. See \cref{remark:Conjecture} for further details.

In the construction of the embedding $g:M\times \S^1 \hookrightarrow \C^3$ of \cref{thm:main_result} we can prescribe the \emph{smooth isotopy class} of the underlying $g_{|M\times \{p\}}: M\hookrightarrow \C^3$. In particular, we can achieve knotting for $g$ by knotting $g_{|M\times \{p\}}$.
\begin{corollary}\label{cor:KnottedCREngelEmbeddings}
Let $M$ be a closed, oriented $3$-manifold. There exist infinitely many embeddings $g_k:M\times\S^1\hookrightarrow \C^3$, $k\in \Z$, that are pairwise non-isotopic and whose complex tangencies $\D^{g_k}_{J_\std}$ are Engel. Moreover, all these Engel structures $\D^{g_k}_{J_\std}$ are Engel homotopic.
\end{corollary}

The Engel structures constructed in \cref{thm:main_thm}, \cref{thm:main_result} and \cref{cor:KnottedCREngelEmbeddings} are small deformations of the \emph{derived prolonged distribution} $\D_\derprol^g$ associated to the convex bundle immersion $g$ (see \cref{def:DerivedProlonged}, \cref{sec:preliminaries}, for further details). That is, we are not able to prescribe exactly the Engel structure that we obtain. This leads to the following open questions:
\begin{itemize}
    \item Suppose $N \subseteq \C^3$ is co-real with Engel complex tangencies $\D$. Can every perturbation of $\D$ be induced by a perturbation of the embedding $N \hookrightarrow \C^3$?
    \item Can the Cartan prolongation of a contact structure \cite{Cartan,Montgomery1999} be realized as the complex tangencies of a co-real submanifold in $\C^3$? 
\end{itemize}

\begin{remark}
    In the case where $N\subseteq \C^3$ is the $4$-torus, regarded as the trivial $\S^1$-bundle over the $3$-torus $M=\S^1\times \S^1 \times \S^1$, we can match the Engel structure defined by the complex tangencies and the derived prolonged distribution. This can be done by choosing a base\nobreakdash-independent (\cref{def:BaseIndependent}) convex embedding $g:N\hookrightarrow TM$, and the Clifford torus $\S^1\times \S^1 \times \S^1\subseteq \C\times\C\times\C=\C^3$ as a totally real embedding in \cref{thm:main_thm}. In this case, the derived prolonged distribution and the complex tangencies match, as proven in \cref{lem:BaseIndependent}.
\end{remark}

\subsection{Outline}

The article is organized as follows. In \cref{sec:preliminaries} we review the relevant background, including: (1) the notion of prolongation, which will be our main recipe to construct Engel structures, (2) the role of convexity as a criterion for Engelness, (3) the construction of tubular neighbourhood models around totally real submanifolds. In \cref{sec:integralZero} we discuss the construction of convex bundle immersions; this uses Saldanha's analysis of the homotopy type of the space of convex curves in the sphere. In \cref{sec:zoooming} we implement our zooming argument, establishing \cref{thm:main_thm}. These ingredients are combined in \cref{sec:proofMain} to prove \cref{thm:main_result} and \cref{cor:KnottedCREngelEmbeddings}.

\textbf{Acknowledgments:} AdP is funded by the NWO grant Vidi.223.118 “Flexibility and rigidity of tangent distributions". During this project, EF received support from an AMS-Simons Travel Grant. WZ is funded by the Spanish FPI predoctoral program PRE2020-092185. WZ further acknowledges support from the ICMAT project PID2022-142024NB-I00, which covered travel and local expenses related to his visits to AdP and EF. He is grateful to the Department of Mathematics at the University of Georgia and the Mathematical Institute of Utrecht University for their hospitality during his visits. Special thanks to T. Vogel, who many years ago told AdP about \cref{ques:main}.

\section{Preliminaries} \label{sec:preliminaries}

Throughout the paper, all manifolds, distributions, and maps are smooth. To avoid orientability issues, we assume all manifolds and bundles are equipped with fixed orientations.

\subsection{Engel structures}

A distribution $\D\subseteq TN$ on a manifold $N$ is a smooth subbundle of its tangent bundle. Given two distributions $\D_1,\D_2$, their \emph{Lie-bracket} is a module of vector fields:
\begin{equation*}
    [\D_1,\D_2] \defi \{[u_1,u_2] \st u_1\in\Gamma(\D_1), u_2\in\Gamma(\D_2) \}.
\end{equation*}
If $[\D_1,\D_2]$ has pointwise constant rank, it defines a distribution; we denote it by the same symbol. Throughout, we always assume that $[\D_1,\D_2]$ has constant rank.

\begin{definition}
Let $N$ be a $4$-manifold and $\D\subseteq TN$ a rank-$2$ distribution. We say that $\D$ is \emph{Engel} if $\D \subsetneq [\D,\D] \subsetneq [\D,[\D,\D]] = TN$.
\end{definition}
The Engel condition is local and $C^2$-open, so it persists under sufficiently small $C^2$ perturbations.

\subsection{Prolonged distributions}

Fix a closed, oriented $3$-manifold $M$ together with an oriented $\S^1$-bundle $\pi:E \to M$. Let $\S TM$ denote the sphere bundle of oriented lines over $M$. We consider the following class of rank-$2$ distributions \cite{Adachi_EngelTrivialCharFoliation, KlukasSahamie2011,delPino2018Prolongations}:
\begin{definition}
A \emph{curve family} over $M$ is a smooth fibred map 
    \begin{equation*}
        f:E\to \S TM.
    \end{equation*}
At each $p\in M$, let $f_p: E_p \to \S(T_pM)$ denote the restriction of $f$ to the fibres over $p$.
    
The \emph{prolonged distribution} associated to a curve family $f$ is the rank-$2$ distribution $\D_\prol^f \subseteq TE$ defined by
    \begin{align*}
        \D_\prol^f(q)\defi (\d\pi_q)^{-1}( L_q ),
    \end{align*}
    where $L_q \subseteq T_{\pi(q)}M$ is the oriented line corresponding to the point $f_p(q)\in \S(T_{\pi(q)}M)$.
\end{definition}

One way to construct a curve family over $M$ is to take the fibrewise derivative of a \emph{bundle immersion} $g: E \looparrowright TM$. More precisely, if $g_p: E_p \looparrowright T_pM$ is an immersion for all $p$, it induces a curve family via the formula
\begin{align*}
    \mathbf{t}(g)_p\defi \mathbf{t}(g_p):E_p& \to \S T_pM,\\
    t &\mapsto [(g_p)'(t)];
\end{align*}
where $[(g_p)'(t)]$ denotes the oriented line represented by $(g_p)'(t)$. The following notion may seem a bit unmotivated at this point, but it will become relevant in our main construction:
\begin{definition}\label{def:DerivedProlonged}
If $g:E \looparrowright TM$ is a bundle immersion, its \emph{derived prolonged distribution} is $\D_\derprol^g\defi\D_\prol^{\mathbf{t}(g)}$.
\end{definition}

\subsection{Engel prolongations}

We now recall a sufficient condition ensuring that a prolonged distribution $\D_\prol^f$ is Engel.
\begin{definition}
    A smooth immersion $\gamma: \S^1 \looparrowright \S^2 \subseteq \R^3$ is said to be  \emph{locally convex} (or simply \emph{convex}) if its geodesic curvature is strictly positive everywhere. That is, 
        \begin{equation*}
            \det(\gamma(t),\gamma'(t),\gamma''(t))>0,\quad \forall\ t\in \S^1,
        \end{equation*}
        with respect to the standard orientation of $\R^3$. 
\end{definition}

A straightforward computation shows that convexity is preserved under the action of $\operatorname{GL}^+(3)$ on $\S^2$. Therefore, the following is well-defined:
\begin{definition}\label{def:convex_der_curves}
A curve family $f: E \to \S TM$ is \emph{convex} if, for each $p \in M$, the curve $f_p: E_p \to \S T_pM$ is convex in the above sense.
    
A bundle immersion $g: E \looparrowright TM$ is \emph{convex} if its fibrewise derivative $\mathbf{t}(g)$ is convex.
\end{definition}

The following criterion will be crucial in our construction of Engel structures. It is a particular case of \cite[Proposition 7]{CPPP}:
\begin{proposition} \label{prop:conv_is_engel}
    Let $f:E\to \S TM$ be a curve family. The following statements hold:
    \begin{itemize}
        \item [(i)] If $f$ is a bundle immersion, then $\D_\prol^f$ is non-integrable.
        \item [(ii)] If $f$ is convex, then $\D_\prol^f$ is Engel.
    \end{itemize}
\end{proposition}
Do note that, according to \cite[Proposition 7]{CPPP}, the first statement in \cref{prop:conv_is_engel} is an if and only if, but the second one is not.

\begin{corollary} \label{cor:ConvexBundleImmersion}
If $g: E \looparrowright TM$ is a convex bundle immersion, then $\D_\derprol^g$ is an Engel distribution on $E$.
\end{corollary}

\subsection{A tubular neighbourhood for totally real submanifolds}

Fix an oriented, closed $n$-dimensional manifold $M$. We write $M \subseteq TM$ for the zero section. Then $T(TM)$ canonically splits along $M$:
\begin{equation*}
    T(TM)|_M \simeq TM \oplus TM =: \mathrm{Hor} \oplus \mathrm{Vert}.
\end{equation*}

This allows us to consider:
\begin{definition}
An almost complex structure $J$ on $TM$ is called \emph{splitting-adapted} if, along the zero section, $J: \mathrm{Hor} \simeq TM \rightarrow \mathrm{Vert} \simeq TM$ is the identity isomorphism.
\end{definition}

Which provides a model for the tubular neighbourhood of middle-dimensional totally real submanifolds:
\begin{lemma}\label{lem:tangent_bundle}
Let $\iota:M^n\looparrowright (X^{2n},J)$ be a totally real immersion (resp. embedding). There exists an immersion (resp. embedding) of the tangent bundle $\varphi: TM \looparrowright X$, such that:
    \begin{enumerate}[label=(\roman*)]
        \item $\varphi|_M = \iota$;
        \item The pullback almost complex structure $\varphi^*(J)$ on $TM$ is splitting-adapted.
    \end{enumerate}
\end{lemma}

\begin{proof}
For notational simplicity assume that $M \subseteq X$ is a submanifold and $\iota$ is thus the inclusion. Since $M$ is totally real, the tangent bundle $TX$ along $M$ splits as $TX|_M = TM \oplus J(TM)$. Choosing an auxiliary Riemannian metric on $X$ yields then the tubular neighbourhood embedding 
    \begin{align*}
        \varphi: TM &\looparrowright (X,J), \\
        q &\longmapsto \exp_p\big(J_p(q)\big),
    \end{align*}
where $p = \operatorname{proj}_{TM}(q)$. It is straightforward to check that $\varphi$ satisfies the conclusions of the lemma.
\end{proof}

One can then take a local chart to deduce that:
\begin{corollary} \label{lem:local_coordinates}
Consider $(TM,J)$ with $J$ splitting-adapted and fix a point $p \in M$. There are base $(x_1,\ldots,x_n)$ and fibre coordinates $(y_1,\ldots,y_n)$ around $p$ such that $J = J_\std$ along the zero section $\R^n_x \subseteq T\R^n_x = \R^n_x \times \R^n_y = \C^n$. Here $J_\std$ denotes the standard complex structure on $\C^n$.
\end{corollary}

\section{Convex curves with zero integral} \label{sec:integralZero}

Our goal in this section is to construct Engel derived prolongations, as in  \cref{cor:ConvexBundleImmersion}. The statement to prove reads:
\begin{proposition} \label{prop:EngelDerivedProlongation}
Let $N$ be an oriented $\S^1$-bundle over a closed, oriented $3$-manifold $M$. Then a convex bundle immersion $g: N \looparrowright TM$ exists. If $N$ is the trivial bundle, the map $g$ can be assumed to be an embedding.
\end{proposition}

We will make use of the following auxiliary result:
\begin{lemma}\label{lem:S1-bundle_immersed_in_STM}
Let $N$ be an oriented $\S^1$-bundle over a closed, oriented $3$-manifold $M$. Then, there exists a bundle immersion $f: N \looparrowright \S TM$. If the Euler class $e(N)$ is even then $f$ can be chosen to be a bundle embedding. 
\end{lemma}
\begin{proof}
    By the fibrewise Hirsch--Smale $h$-principle, it suffices to produce a bundle monomorphism $V(N)\rightarrow V(\S TM)$ between the associated vertical bundles of $N$ and $ \S TM$ respectively (see \cite{cieliebak2024introduction}). The existence of such a monomorphism follows from the triviality of $V(N)$ and $\S TM\cong M\times \S^2\subseteq M\times \R^3$, since $N$ is oriented and $M$ is parallelisable \cite{Benedetti-LiscaFraming3fold,milnor1974characteristic}. Indeed, if we denote by $\pi:N\rightarrow M$ the bundle projection and we fix a global trivialisation $V(N)\cong N\times \R$, then a bundle monomorphism is given by 
    \begin{align*}
        F: N\times \R &\rightarrow V(M\times \S^2),\\
        (p,t) &\mapsto (\pi(p), e_1, t\cdot e_2).
    \end{align*}

    For the embedding case observe that, since the Euler class $e(N)$ is even, $N$ is isomorphic to the bundle of oriented lines of a plane distribution on $M$, so the result follows (see for instance \cite[Proposition 4.3.2]{geiges2008introduction}).
\end{proof}

Our goal is to deform the bundle immersions $f$ of \cref{lem:S1-bundle_immersed_in_STM} to make each $f_p$ convex and then integrate. This will yield the claimed convex bundle immersions $g$ of \cref{prop:EngelDerivedProlongation}. Concretely, this amounts to the construction of curves in $\R^3$ whose (normalised) derivatives trace convex curves in $\S^2$. For a single curve this can be done manually: you pick a convex curve and you integrate, ensuring that the total integral is zero, so the primitive curve is also closed. However, when dealing with families, it is more handy to refer to Saldanha \cite{saldanhaHomotopyConvex}.

\subsection{Saldanha's result}

In what follows, we denote $I=[0,1]$ and think of the unit circle as $\S^1=[0,1]/_{0\sim1}$, so that the open interval $(0,1)\subseteq \S^1$ is naturally embedded therein. Write $\Imm$ for the space of immersions $\S^1 \looparrowright\S^2$. Following \cite{saldanhaHomotopyConvex}, call a convex curve \emph{complicated} if it is not a concatenation of embedded convex loops, and write $\Conv^c\subseteq \Imm$ for the subspace of such curves. The following $h$-principle type result holds:
\begin{proposition} {\cite[Proposition 1.4]{saldanhaHomotopyConvex}}\label{prop:conv_h_principle}
    Let $K$ be a compact CW-complex and $A\subseteq K$ a closed CW-subcomplex. For every continuous map $f: (K,\Op(A)) \to (\Imm,\Conv^c)$ there exists a homotopy $H_s: (K,\Op(A)) \to (\Imm,\Conv^c)$, $s\in[0,1]$, such that:
    \begin{enumerate}[label=(\roman*)]
        \item $H_0 =f,$
        \item $H_s(k) = f(k)$ for all $(s,k) \in [0,1]\times \Op(A)$,
        \item $\operatorname{Image}(H_1) \subseteq \Conv^c$.
    \end{enumerate}
\end{proposition}

\subsection{The zero integral condition}
For our purposes we require a refinement of \cref{prop:conv_h_principle} to produce families of convex curves with vanishing integral.

A necessary condition for a smooth curve in $\R^3$ to have integral zero is that the convex hull of its image contains the origin $0\in\R^3$. A stronger hypothesis, coming from control theory and introduced in the context of convex integration by Gromov, reads: A smooth curve $\gamma:\S^1\to \S^2$ \emph{strictly surrounds} the origin if
\begin{equation*}
    \Op (0) \subseteq \operatorname{Conv}(\img \gamma).
\end{equation*}
Equivalently, there is no non-zero vector $u\in \R^3$ with $u \cdot \gamma(t) \geq 0$ for all $t\in \S^1$. Geometrically, this means that $\img\gamma$ is not contained in any closed hemisphere.

The following allows us to achieve the zero integral condition:
\begin{lemma}[Gromov's reparametrisation lemma {\cite[Subsection 2.4.1]{gromovPartialDifferentialRelations1986}}] \label{lem:Gromov_reparam}
    Let $K$ be a compact CW-complex and $A\subseteq K$ a closed CW-subcomplex. Consider a $K$-parameter family of immersions $f:K\to \Imm$ such that
    \begin{enumerate}[label=(\roman*)]
        \item $f(k)$ strictly surrounds the origin for all $k\in K$;
        \item $\displaystyle \int_0^1 f(k)(t) \d t= 0\in \R^3$ for $k\in \Op (A)$.
    \end{enumerate}
    Then, there exists a $K$-parameter family of diffeomorphisms  $\psi:(K,\Op(A)) \to (\Diff(\S^1),\Id)$, such that the reparametrisations $\tilde{f}(k) \defi f(k)\circ \psi(k)$ satisfy
    \begin{equation*}
        \int_0^1 \tilde{f}(k)(t) \d t = 0, \quad \text{for all } k\in K.
    \end{equation*}
\end{lemma}

We will show that the strictly surrounding condition can be achieved for any family of complicated curves by a suitable homotopy. For this we introduce a convenient subclass.

\begin{definition}
    Let $\gamma:[0,1]\to \S^2$ be a convex curve. We say that:
    \begin{itemize}
        \item $\gamma$ has a \emph{wiggle} in a sub-interval $[a,b]$, if $\gamma(a)=\gamma(b)$ and $\gamma|_{[a,b]}$ is a smooth, closed, embedded convex curve after identifying the endpoints.
        \item $\gamma$ has an \emph{$n$-wiggle} in $[a,b]$, if there exists a partition $a=t_0<t_1<\cdots<t_n=b$ such that $\gamma$ has a wiggle in $[t_i,t_{i+1}]$ for all $i=0,\ldots,n-1$, and all arcs $\gamma|_{[t_i,t_{i+1}]}$ have the same image.
        \item $\gamma$ \emph{$n$-completely surrounds} the origin if there exist disjoint intervals $[a_1,b_1]$ and $[a_2,b_2]$ such that $\gamma$ has a $n$-wiggle in each $[a_j,b_j]$, $j=1,2$, and
        \begin{equation*}
            \Op (0)\subseteq \operatorname{Conv}(\img \gamma|_{[a_1,b_1]\cup[a_2,b_2]}).
        \end{equation*}
        In particular, $\gamma$ strictly surrounds $0$.
    \end{itemize}
    \vspace{-8mm}
\end{definition}
The subspace of such curves is denoted by $\Conv^c_n\subseteq \Conv^c$. In view of \cref{lem:Gromov_reparam}, it is natural to deform families of complicated convex curves so that they $n$–completely surround the origin, for a fixed $n\ge 1$. This can be achieved by \emph{grafting}; the precise relative statement is as follows.
\begin{proposition}\label{prop:complicated_can_completely_surrounded}
    Let $K$ be a compact CW-complex and $A\subseteq K$ a closed CW-subcomplex. For every continuous map $f:(K,\Op(A)) \to (\Conv^c,\Conv^c_{n+2})$ with $n\geq1 $, there exists a homotopy $H_s: (K,\Op (A)) \to (\Conv^c,\Conv^c_{{n+2}})$, $s\in[0,1]$, such that
    \begin{enumerate}[label=(\roman*)]
        \item  $H_0=f$,
        \item $H_s(k) = f(k)$ for all $(s,k)\in [0,1]\times \Op (A)$,
        \item $\operatorname{Image}(H_1) \subseteq \Conv^c_n$.
    \end{enumerate}
\end{proposition}
The remainder of the section is devoted to its proof, which in turn implies \cref{prop:EngelDerivedProlongation}.

\subsection{Proof of \cref{prop:complicated_can_completely_surrounded} via grafting}

A \emph{graftable model} is a convex curve $\gamma_1:[t_0,t_1]\to \S^2$ that is tangent to some latitude circles $\{z= C_i\}$ at $t_i$, with $C_0>0>C_1$, and whose orientation matches the convex orientation of those circles. Consider then a convex curve $\gamma:\S^1\rightarrow \S^2$ that contains a graftable model $\gamma_1=\gamma_{|[t_0,t_1]}$. The \emph{grafting} of $\gamma$ along $[t_0,t_1]$ is the homotopy of convex curves obtained by cutting $\gamma$ at $t_0$ and $t_1$,  rotating the arc $\gamma_1=\gamma_{|[t_0,t_1]}$ about the $z$-axis by angle $2\pi s$ for $s\in[0,1]$, and reconnecting it along the latitude circles $z=C_i$ to the complementary arc $\gamma|_{\S^1\setminus [t_0,t_1]}$. For $s=1$, the resulting curve contains two disjoint $1$-wiggles whose convex hull contains the origin. In other words, the resulting convex curve\footnote{Technically, one needs to smooth out the homotopy constructed in this manner. This can be done in a $C^1$-small manner, supported in a small neighbourhood of the cutting points, while preserving convexity. This will be implicit in the remainder of the discussion. } $1$-completely surrounds the origin. See \cref{fig:grafting} for a schematic depiction of the grafting homotopy. We also define the \emph{$n$-grafting} by repeating the grafting homotopy $n$ times, equivalently, rotate by angle $2n\pi s$; in this case, the resulting curve $n$-completely surrounds the origin. 

\begin{figure}[ht]
    \centering
    \includegraphics[width=0.5\textwidth]{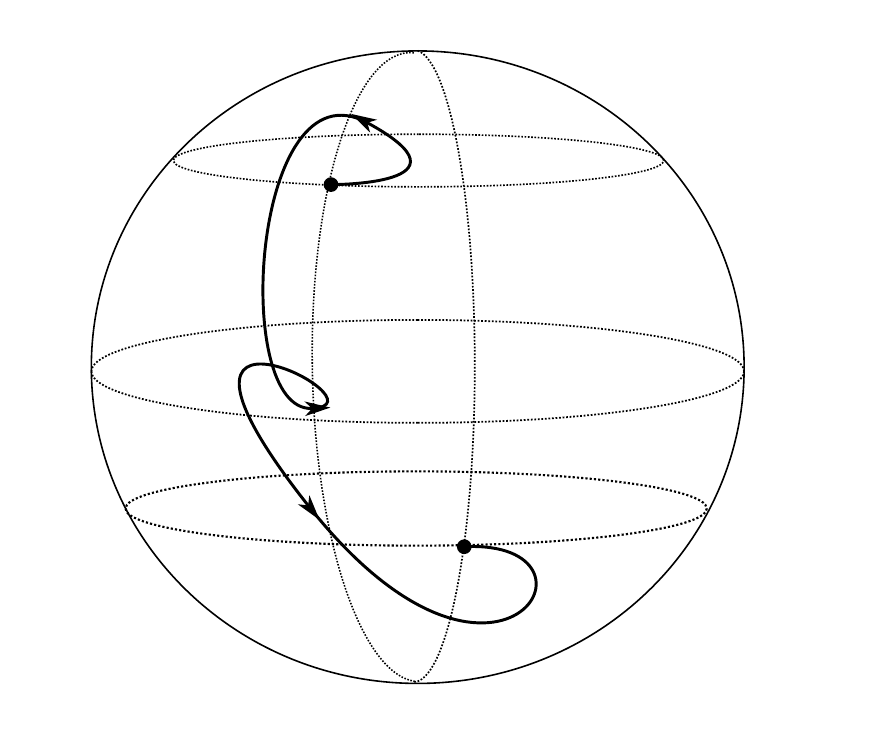}
    \caption{A graftable arc $\gamma_1:[t_0,t_1]\to\S^2$. At its endpoints, the arc is tangent to two latitude circles.}
    \label{fig:graftable}
\end{figure}

\begin{figure}[ht]
    \centering
    \includegraphics[width=1.1\textwidth]{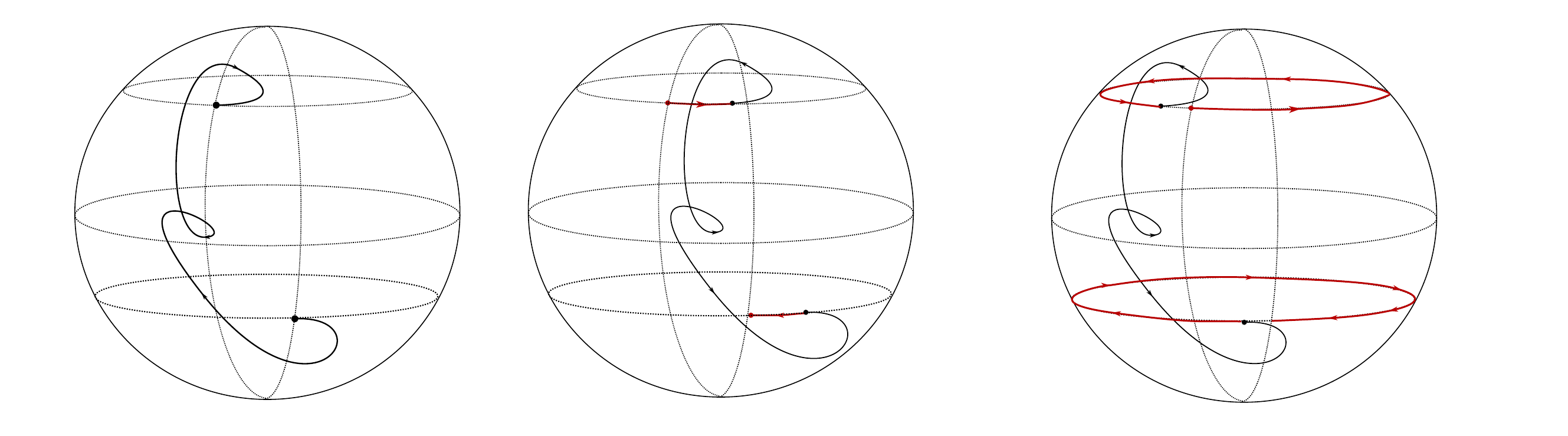}
    \caption{Grafting homotopy of a graftable arc. The images depict the homotopy at times $s=0, 0.2, 0.8$, from left to right. In particular, in the last frame we see that the grafting has almost completed a full loop, which will create a wiggle in the northern hemisphere and another wiggle in the southern hemisphere (both in red).}
    \label{fig:grafting}
\end{figure}

The following lemma is implicit in Saldanha’s work:
\begin{lemma}\label{lem:Complicated_loose}
    Let $f:K\to \Conv^c$ be a family of complicated convex curves. Then, there are:
    \begin{itemize}
        \item Compact subsets $K_1,\ldots,K_m\subseteq K$  covering $K$.
        \item Compact subsets $K_{j,1},K_{j,2}\subseteq K_j$ covering $K_j$, for each $j$.
        \item Continuous functions $a_j,b_j,c_j,d_j:K_j\to(0,1)\subseteq \S^1$, for each $j$, with $a_j<b_j<c_j<d_j$.
        \item Intervals $I_j(k) \defi [b_j(k),c_j(k)] \subseteq \tilde{I}_j(k)\defi [a_j(k),d_j(k)]$.
    \end{itemize}
Such that:
    \begin{enumerate}[label=(\roman*)]
        \item For every $k\in K_i \cap K_j$ and $i\neq j$, it holds that $\tilde{I}_j(k)\cap \tilde{I}_i(k)=\emptyset$.
        \item Over $K_{j,1}$, for each $j$, the arcs $f(k)\!\mid_{I_j(k)}$ are \emph{projectively equivalent}\footnote{That is, there exists a continuous map $k\mapsto P_k\in\mathrm{PGL}^+(3,\R)$ such that $P_k\circ f(k)\!\mid_{I_j(k)}$ is a graftable model.} to a graftable model. 
        \item Over $K_{j,2}$, for each $j$, the arcs $f(k)\!\mid_{\tilde I_j(k)}$ are projectively equivalent to a graftable model.
        \item Suppose there is $k\in K_j$ such that $f(k)$ has an $n$–wiggle in some $[a,b]\subseteq (0,1)$ with $n\ge 2$. Then $\tilde I_j(k)$ is not fully contained in $[a,b]$. Moreover, $[a,b] \cap \tilde I_j(k)$ contains no wiggle.
    \end{enumerate}
\end{lemma}

\begin{proof}
    Given a convex curve $\gamma\colon [0,1]\to \S^2$ and $t_0\in[0,1)$, the \emph{next step} $\mathrm{ns}_\gamma(t_0)$ is the smallest $t_0^+>t_0$ such that $\gamma|_{[t_0,t_0^+]}$ does not extend to an embedded convex loop. There are multiple reasons why the curve may not extend. Saldanha classifies the possibilities; see the first two paragraphs of \cite[Sec.~9]{saldanhaHomotopyConvex}. A step is \emph{bad} if $\gamma|_{[t_0,t_0^+]}$ is a wiggle. Otherwise it is \emph{good}. 

    If $t_1=\mathrm{ns}_\gamma(t_0)$ and $[t_0,t_1]$ is a good step, then for any $\varepsilon>0$ there exist subintervals
    \begin{equation*}
        [t_0,t_1]\subsetneq [\widehat{t_0},\widehat{t_1}]\subsetneq [\widetilde{t_0},\widetilde{t_1}] \subsetneq [t_0-\varepsilon ,t_1+\varepsilon],
    \end{equation*}
    such that at least one of $\gamma|_{[\widehat{t_0},\widehat{t_1}]}$ or $\gamma|_{[\widetilde{t_0},\widetilde{t_1}]}$ is projectively equivalent to a graftable model. Moreover, this holds parametrically in families of good steps, see \cite[Lemma 8.2, Lemma 9.1]{saldanhaHomotopyConvex}.

    Consider now the sequence of next steps $\{\mathrm{ns}_{f(k)}^j(0)\}$, for each $j$ and $k$. Given the family $f:K\to \Conv^c$, we can define the subspace $K_j\subseteq K$ consisting of those $k\in K$ such that the $j$th-step $S_j(k) :=[\mathrm{ns}_{f(k)}^j(0),\mathrm{ns}_{f(k)}^{j+1}(0)] \subseteq [0,1)$ is good. According to the definition of bad step, every complicated curve has at least one good step, so $\{K_j\}$ is a cover of $K$. The intervals produced in this manner, for different $K_j$ and $K_i$, are not disjoint. However, being a good step is an open condition. This allows us to shift all the $S_j(k)$ to the right to make them disjoint; the definition of good step then defines the claimed $I_j(k) \subseteq \tilde{I}_j(k)$. For details, see the first three paragraphs of the proof of \cite[lemma 10.1]{saldanhaHomotopyConvex}.

    For $(iv)$, consider $\gamma:[0,1]\to \S^2$ with an $n$-wiggle in $[a,b]$ with partition $a=s_0<\cdots <s_n=b$. Then, for any $s_i\leq t\leq s_{i+1}$ with $i=0,\ldots,n-2$, it holds that $\mathrm{ns}_\gamma(t)=t^+$ satisfies $s_{i+1}\leq t^+\leq s_{i+2}$. Moreover, $[t,t^+]$ is wiggle and thus a bad step. Hence, any good step must meet $[a,b]$ near an endpoint, i.e., inside of $[s_0,s_1)$ or $(s_{n-1},s_n]$. Since $\tilde I_j(k)$ is a small thickening of a good step, $(iv)$ follows.
\end{proof}

\begin{proof}[Proof of \cref{prop:complicated_can_completely_surrounded}]
    Our proof follows Saldanha’s approach \cite[Lemma~10.1]{saldanhaHomotopyConvex}. We first treat the absolute case \(A=\varnothing\), since $f$ is a family of complicated convex curves, we use the notation from \cref{lem:Complicated_loose}.

    Choose smaller compact subsets $W_{j,1}\Subset K_{j,1}$ and $W_{j,2}\Subset K_{j,2}$ for $j=1,\ldots,m$, such that
    \begin{equation*}
        K = \bigcup_{j=1}^m \operatorname{Int}(W_{j,1})\cup\operatorname{Int}(W_{j,2}).
    \end{equation*}
    We construct the homotopy inductively for $j=1,\ldots,m$, modifying the family only over $K_j$. 

    Set $f_0:=f$. On $K_{1,1}$ perform an $n$–grafting of $f_0(k)$ along $I_1(k)$ for $k\in W_{1,1}$, obtaining $\tilde f_0$, and keep $\tilde f_0(k)=f_0(k)$ for $k\notin K_{1,1}$. On $K_{1,1}\setminus W_{1,1}$ interpolate with a bump function $\beta\colon K\to[0,1]$ satisfying $\beta\equiv 1$ on $W_{1,1}$ and $\beta\equiv 0$ on $K\setminus K_{1,1}$.

    Next, on $K_{1,2}$ perform an $n$-grafting of $\tilde f_0(k)$ along $\tilde{I}_1(k)$ for $k\in W_{1,2}$, producing $f_1$, and set $f_1(k)=\tilde f_0(k)$ for $k\notin K_{1,2}$; interpolate on $K_{1,2}\setminus W_{1,2}$ with another cutoff as before. The pair of $n$-wiggles created in the first step lies inside $\tilde I_1(k)$, and the grafting homotopy (rotation about the $z$-axis) preserves the $n$-completely surrounding condition. This yields a homotopy from $f_0$ to $f_1$ supported in $K_1$, with $f_1(k)$ $n$-completely surrounding the origin for $k\in W_{1,1}\cup W_{1,2}$.
    
    Iterating this two-step procedure for $j=2,\ldots,m$ yields maps $f_2,\ldots,f_m$. Because the intervals  $\tilde I_j(k)$ are pairwise disjoint, each stage preserves the previously created $n$-wiggles. Consequently, for every $k\in K$, the curve $f_m(k)$ $n$-completely surrounds the origin.
    
    For the relative case, fix a compact subset $W\subseteq \Op(A)$ such that $A\subseteq \operatorname{Int}(W)$. We perform the construction above with cutoffs so that the entire homotopy is supported in $K\setminus W$. Consider $k\in \Op(A)\setminus W$ and denote by $B_k$ and $B_k'$ the two intervals supporting the $(n\!+\!2)$\nobreakdash-wiggles. According to \cref{lem:Complicated_loose} (iv), each $\tilde{I}_j(k)$ can meet $B_k$ or $B_k'$ only along a subsegment that is itself not a wiggle. Since $k$ may be in multiple $K_j$, both ends of $B_k$ or $B_k'$ may intersect one such thickened good step.
    
    Consequently, since the grafting is supported inside $\bigcup_j \tilde I_j(k)$, it can remove at most one wiggle from each side of $B_k$ or $B_k'$. It follows that at least $n$ wiggles remain from each. We deduce that the $n$-completely surrounding condition holds on the collar $\Op(A)\setminus W$, and thus globally, while the homotopy remains equal to $f$ on $W$.
\end{proof}

\subsection{Proof of \cref{prop:EngelDerivedProlongation}}

Since $M$ is a closed, oriented $3$-manifold, we can fix a trivialisation $TM\simeq M\times \R^3$ for the rest of the proof. 

Assume first that $N\cong M\times \S^1$ is trivial. Fix an immersed curve $\gamma:\S^1 \looparrowright\S^2$ and apply \cref{prop:conv_h_principle} to homotope it to a complicated convex curve. Apply then \cref{prop:complicated_can_completely_surrounded} to homotope it further so it becomes completely surrounding. Lastly, apply Gromov's reparametrisation \cref{lem:Gromov_reparam} so its integral is zero. Denote the resulting curve by $\eta: \S^1 \looparrowright \S^2$. Then $\eta$ has a primitive $\nu: \S^1 \looparrowright \R^3$. By transversality, after a $C^3$–small perturbation, we may assume that $\nu$ is embedded while $\mathbf{t}(\nu)$ remains convex. Define a convex bundle embedding by
\begin{equation*}
    g:M\times \S^1 \hookrightarrow M\times \R^3\cong TM;\quad (p,t)\mapsto(p,\nu(t)).
\end{equation*}
This concludes the argument in the trivial bundle case.

For a general bundle $N$ we proceed as follows. Fix a finite cover by compact subsets $\{K_j\}_{j=1,\ldots,m}$ of $M$ over which $N$ trivialises, and a smaller cover by compact sets $\{W_j\}_{j=1,\ldots,m}$ with $W_j\Subset K_j$. Without loss of generality, assume that the transition maps $\sigma_{ij}:K_i\cap K_j\times \S^1 \to K_i\cap K_j\times \S^1$ are valued in $\operatorname{SO}(2)$, and that the refined cover $\{W_j\}$ has order $3$, so any point of $M$ lies in at most $4$ of the $W_j$ (see \cite[Theorem V.1]{hurewicz2015dimension}).

By \cref{lem:S1-bundle_immersed_in_STM}, there exists a bundle immersion $f: N \looparrowright \S TM\simeq M\times \S^2\subseteq M\times \R^3$. This immersion can be described as a collection of continuous maps 
\begin{equation*}
    f_j:K_j\longrightarrow \Imm,\qquad j=1,\ldots,m,
\end{equation*}
which are compatible with the transition maps in the obvious sense: on $K_i\cap K_j$ one has $f_i(k)\circ\sigma_{ij}(k)=f_j(k)$.

Our goal is to deform these maps so that their values lie in the subspace of convex curves with zero total integral. Notice that the latter condition is independent of the chosen trivialisations, because the transition functions are $\operatorname{SO}(2)$-valued. After this is achieved, the desired convex bundle immersion is obtained by fibrewise integration.

To achieve this, proceed inductively over the compact sets $W_j\Subset K_j$, $j=1,\ldots, m$. Set $n\defi 2\cdot4+1 =9$ which will be used in \cref{prop:complicated_can_completely_surrounded}.

For the first chart, apply \cref{prop:conv_h_principle}, \cref{prop:complicated_can_completely_surrounded} and \cref{lem:Gromov_reparam} (in the absolute case) to the map $f_1:K_1\to \Imm$. This yields a homotopy $f_1^s:K_1\to \Imm$, $s\in[0,1]$, with $f_1^0=f_1$ and $f_1^1:K_1\to \Conv^c_n$ having zero integral. We cut off this deformation outside of $W_1$: we fix a bump function $\rho:K_1\to[0,1]$, supported in $\Op(W_1)$, with $\rho_{|W_1}\equiv 1$, and define $\tilde{f}_1(k) \defi f^{\rho(k)}_1(k)$. Since the deformation is supported in $\Op(W_1)$ it defines a global bundle immersion $\tilde{f}:N \looparrowright \S TM$ with the required properties over $W_1$.

Since \cref{prop:conv_h_principle}, \cref{prop:complicated_can_completely_surrounded} and \cref{lem:Gromov_reparam} admit parametric and relative versions, the same argument can be iterated over $W_2,\ldots,W_m$, keeping the previously treated region fixed at each step. Since every point of $M$ lies in at most four of the sets $W_j$, the multiplicity of the wiggles may decrease by at most $2$ per step, hence by at most $8$ in total. It follows that, if we start from $\Conv^c_n$ with $n=9$, the final family will still lie in $\Conv^c_{n-8}=\Conv^c_1$, so the strictly surrounding condition persists throughout the induction. I.e., we have obtained a bundle immersion $\hat f: N \looparrowright \S TM$ with local representatives
\begin{equation*}
    \hat{f}_j: K_j\longrightarrow \Conv^c_{n-8}\subseteq \Imm,\qquad j=1,\ldots,m.
\end{equation*}
We may assume that these curves have zero total integral after an application of \cref{lem:Gromov_reparam}.

The final step of the proof amounts to integrating the family $\hat f$. Define local maps
\begin{equation*}
    g_j\colon K_j\longrightarrow \Imm(\S^1,\R^3), \qquad
    g_j(k)(t)\coloneqq \int_0^t \hat f_j(k)(\tau)\,d\tau,\qquad j=1,\ldots,m,
\end{equation*}
by fibrewise integration. Since $\int_0^1 \hat f_j(k)(\tau)\,d\tau=0$, each $g_j(k)$ is a closed curve in $\R^3$.

On overlaps $K_i\cap K_j$ the families $g_i$ and $g_j$ may differ by a translation and thus not define a bundle map. Namely, if we write
\begin{equation*}
    \sigma_{ij}(k)\colon \S^1\to \S^1,\ t\mapsto t-\delta(k),
\end{equation*}
for the transition function, then we get:
\begin{equation*}
    \begin{aligned}
    (g_i\circ \sigma_{ij})(k)(t)
    &= \int_0^{t-\delta(k)} \hat f_i(k)(\tau)\,d\tau
     = \int_0^{t-\delta(k)} \hat f_j(k)(\tau+\delta(k)\,d\tau \\
    &= \int_{\delta(k)}^{t} \hat f_j(k)(\tau)\,d\tau
     = g_j(k)(t) \;-\; \underbrace{\int_0^{\delta(k)} \hat f_j(k)(\tau)\,d\tau}_{=: \,v_{ij}(k)}.
    \end{aligned}
\end{equation*}
I.e., the translation vector $v_{ij}\colon K_i\cap K_j\to \R^3$ is independent of $t$, as expected.

Since $\R^3$ is contractible, we can now shift the $\{g_j\}$ to match them up. Concretely: choose a partition of unity $\{\phi_j\}$ subordinate to $\{K_j\}$ and, define a bundle map $g:N\to TM \simeq M\times \R^3$ by
\begin{equation*}
    g(q) = \sum_{j=1}^m \phi_j(\pi(q))g_j(\pi(q)),\quad q\in N,
\end{equation*}
where $\pi:N\to M$ is the bundle projection.
In the trivialisation over \(K_j\) (write \(k=\pi(q)\)) we have, using $(g_i\circ\sigma_{ij})(k)=g_j(k)-v_{ij}(k)$,
\begin{equation*}
    g|_{K_j}(k) = \phi_j(k)g_j(k) +\sum_{i\neq j} \phi_i(k) (g_i\circ\sigma_{ij})(k) =g_j - \sum_{i\neq j} \phi_i(k)v_{ij}(k).
\end{equation*}
The correction term is independent of \(t\in\S^1\), hence the fibrewise derivative satisfies $\mathbf t(g)|_{K_j}=\mathbf{t}(g_j)=\hat f_j$, which is convex by construction. Therefore $g$ is a well-defined bundle immersion with convex fibrewise derivative. \hfill\qedsymbol

\section{Zoom-in argument: proof of \cref{thm:main_thm}}  \label{sec:zoooming}

Let $N$ be an oriented $\S^1$-bundle over the closed, oriented $3$-manifold $M$ and let $g: N \looparrowright TM$ be a convex bundle immersion. According to the local model \cref{lem:tangent_bundle}, we may assume that $(X,J) = (TM,J)$, with $J$ splitting-adapted and $M$ embedded as the zero section. 

For $\lambda>0$, consider the fibrewise dilations
\begin{equation*}
    \delta_\lambda: TM \to TM,
\end{equation*}
which scale each fibre by $\lambda$, and define a one-parameter family of convex bundle immersions by
\begin{equation*}
    g^\lambda \defi \delta_\lambda \circ g: N\looparrowright(TM,J).
\end{equation*}

The result will follow if we show that $\D^{g^\lambda}_J$ is Engel for all sufficiently small $\lambda > 0$. To this end, we prove the following claim.
\begin{claim} \label{claim:key}
    For every point $p\in M$, there exists $\lambda(p) > 0$ sufficiently small and a neighbourhood $U(p)$ of $E_p$ such that $\D^{g^\lambda}_J$ is Engel on $U(p)$ for all $0 < \lambda < \lambda(p)$.
\end{claim}
Once the claim is established, \cref{thm:main_thm} follows immediately by a standard compactness argument.

Since the claim is local, choose coordinates as in \cref{lem:local_coordinates}:
\begin{align*}
    \R^3_x \times \R^3_y &\rightarrow T\R^3_x\\
    (x_i,y_i) &\mapsto \sum y_i\cdot \partial_{x_i},
\end{align*}
with $J$ splitting-adapted along $\R^3_x$ and satisfying $J = J_\std$ along the zero section. In these coordinates, the family of convex bundle immersions has the form
\begin{align*}
    g^\lambda:\R^3_x\times \S^1&\looparrowright \R^3_x\times\R^3_y  \\
    (x,t) &\mapsto \big(x, \lambda \gamma(x,t)\big),
\end{align*}
for some $\gamma: \R^3_x \times \S^1 \to \R^3_y$. 

We now introduce a special case of bundle immersion that will be important in the proof of the claim.
\begin{definition}\label{def:BaseIndependent} 
A bundle immersion $g(x,t) = (x,\gamma(x,t)): \R_x^3 \times \S^1\looparrowright \R^3_x\times \R^3_y$ is called \emph{base-independent} if $\gamma$ does not depend on $x$. Equivalently, $\frac{\partial}{\partial x}\gamma \equiv 0$. 
\end{definition}

\begin{proof}[Proof of \cref{claim:key} and \cref{thm:main_thm}] 
We divide the proof into two steps. First, we treat the case where $g$ is base-independent and $J \equiv J_\std$ is the standard complex structure. Then, we reduce the general case to this one by using a \emph{zooming-in} argument.

\underline{Step I: Base-independent case.} We observe:
\begin{lemma}\label{lem:BaseIndependent}
    If $g$ is base-independent and $J \equiv J_\std$ on $\C^3 = \R^3_x \times \R^3_y$, then
    \begin{equation*}
        \D^g_\derprol=\D_{J_\std}^g
    \end{equation*}
    on $\R^3_x\times \S^1$.
\end{lemma}
\begin{proof}
By hypothesis, $g(x,t) = (x,\gamma(t))$ so the differential of $g$ is the map $
g_* = \operatorname{Id}_{T\R^3_x} \times (\gamma)_*$. I.e., the image of $g_*$ at $(x,t)$ is $T\R^3_x \oplus \langle \gamma'(t) \rangle \subseteq T\R^3_x \times T\R^3_y$. The element $J_\std(\gamma'(t))$ is contained in $T\R^3_x$ and is identified with $-\gamma'(t)$ under the canonical identification $T\R^3_x \simeq T\R^3_y$. It follows that the complex tangency, pulled back to the domain, is $\langle \gamma'(t),\partial_t \rangle$. This is exactly how $\D_\derprol^g(x,t)$ was defined.
\end{proof}
In particular, since $g$ was chosen to be a convex bundle immersion, we deduce (\cref{prop:conv_is_engel}) that $\D_{J_\std}^g = \D_\derprol^g$ is Engel in the base-independent, $J = J_\std$ case.

\underline{Step II: The zooming argument.} We now reduce the general claim to the previous case. At an informal level, our argument is based on the fact that rescaling the base by $\lambda\to 0$ makes the bundle immersion closer to being base-independent, whereas rescaling base and fibre simultaneously forces the pulled-back almost complex structure to converge to the standard one.

Consider the family of ``zooming'' coordinate changes:
\begin{align*}
    \Delta_\mu: \R^3_a \times \R^3_b &\to \R^3_x\times \R^3_y,\\
    (a,b)&\mapsto (\mu a,\mu b):
\end{align*}
for $\mu>0$. 

In the coordinates given by $\Delta_\lambda$ (i.e., $\mu=\lambda)$, the bundle immersion $g^\lambda$ becomes:
\begin{equation*}
    \tilde{g}^\lambda(a,t) =(\Delta_\lambda)^{-1}g^\lambda(\lambda a,t ) = (a, \gamma(\lambda a,t)).
\end{equation*}
while the almost complex structure $J$ transforms as:
\begin{equation*}
J^\lambda_{(a,b)}(V) := (\Delta_\lambda)^*J (V)= (\Delta_\lambda)^{-1}\big(J_{(\lambda a,\lambda b)}(\lambda V)\big) =  J_{(\lambda a,\lambda b)}(V),\quad \forall\ V\in \R^3_a\times \R^3_b,
\end{equation*}
which remains splitting-adapted by construction.

\begin{lemma}
The pairs $(\tilde{g}^\lambda, J^\lambda)$ converge to $(\tilde{g}^0,J_\std)$ in the weak $C^\infty$-topology, where $\tilde{g}^0(x,t)=(x,\gamma(0,t))$ is a base-independent bundle immersion. 
\end{lemma}
\begin{proof}
We compute
    \begin{equation*}
        (\tilde{g}^\lambda - \tilde{g}^0) (a,t) = (0,\gamma(\lambda a,t)-\gamma(0,t)).
    \end{equation*}
Expanding $\gamma(\lambda a,t)$ in the first variable around $0$ via Taylor's theorem yields the desired convergence. The convergence of $J^\lambda$ to $J_\std$ follows similarly.
\end{proof}

By the definition of the weak Whitney $C^\infty$ topology, $\D^{\tilde{g}^\lambda}_{J^\lambda}$ is $C^\infty$-close to $\D^{\tilde{g}^0}_{J_\std}$ near the origin. Since $\tilde{g}^0$ is base-independent, Step~I implies that $\D^{\tilde{g}^0}_{J_\std}$ is Engel. Because the Engel condition is $C^2$-open, $\D^{\tilde{g}^\lambda}_{J^\lambda}$ is also Engel for $\lambda$ small enough. This proves \cref{claim:key}, and therefore \cref{thm:main_thm}.
\end{proof}

\section{Proof of \cref{thm:main_result} and \cref{cor:KnottedCREngelEmbeddings}} \label{sec:proofMain}

The last ingredient we need is Gromov's $h$-principle for totally real embeddings, combined with the fact, due to Forstnerič \cite{forstnericTotallyRealEmbeddings1986, gromovPartialDifferentialRelations1986}, that any embedding of a $3$-manifold into $\C^3$ is formally totally real.
\begin{proposition} \label{thm:every_3fold_TR_emb}
Any embedding $\iota: M \hookrightarrow \C^3$ is isotopic, through embeddings, to a totally real embedding.
\end{proposition}

\begin{proof}[Proof of \cref{thm:main_result}]
Let $N$ be an oriented $\S^1$-bundle over a closed, oriented $3$-manifold $M$. By \cref{thm:every_3fold_TR_emb}, we can find a totally real embedding $M \hookrightarrow \C^3$ in any given isotopy class. According to \cref{prop:EngelDerivedProlongation} we can construct a convex bundle immersion $g:N \looparrowright TM$ (an embedding when $N$ is trivial). Then \cref{thm:main_thm} applies, yielding the desired co-real immersion whose complex tangencies are Engel.
\end{proof}

\begin{remark}\label{remark:Conjecture}
When $e(N)$ is even, \cref{lem:S1-bundle_immersed_in_STM} provides a bundle embedding $g:N\hookrightarrow TM$. We conjecture that a convex bundle embedding also exists. This would produce embeddings of more general $\S^1$-bundles in \cref{thm:main_result}, and is left as an open question. Alternatively, one may produce a convex bundle immersion and try to achieve the fibrewise embedding condition.
\end{remark}

\begin{proof}[Proof of \cref{cor:KnottedCREngelEmbeddings}]
    The proof follows by applying the construction above to an infinite family of totally real embeddings $i_k:M^3\hookrightarrow \C^3$, $k\in \Z$, pairwise not smoothly isotopic.

    The existence of such a family relies on the works of Haefliger and Skopenkov \cite{Haefliger1,Haefliger2,Skopenkov}. Indeed, in view of the $h$-principle from \cref{thm:every_3fold_TR_emb} it is enough to find an infinite family of smooth embeddings pairwise not smoothly isotopic to conclude. In the case $M=\S^3$, the existence of a family of embeddings $i_k:\S^3\hookrightarrow \C^3$, $k\in \Z$, pairwise not smoothly isotopic follows from \cite{Haefliger1,Haefliger2}. This proves the result for $M=\S^3$. 
    
    For the general case, consider a smooth embedding $f:M^3\hookrightarrow \R^5\subseteq \C^3$, which exists because of \cite{Wall}. We may assume that $f$ is disjoint from all Haefliger's spheres $i_k$. Then, the embeddings $f\#i_k:M\hookrightarrow \C^3$, $k\in \Z$, obtained by ambient connected sum of $f$ and $i_k$, are pairwise not smoothly isotopic as proved in \cite{Skopenkov}. This concludes the argument. 
\end{proof}

\bibliographystyle{alpha}
\bibliography{references.bib}

\end{document}